\documentclass[a4paper, oneside]{article}

\usepackage{amsmath}					%
\usepackage{amssymb}					%
\usepackage{amsthm}					%
\usepackage{appendix}					%
\usepackage{xcolor}						%
\usepackage{enumerate}				%
\usepackage{hyperref}					%
\usepackage{mathrsfs}					%
\usepackage{pdfsync}					%
\usepackage{titlesec}					%

\titleformat{\section}[hang]							%
{\bfseries\large}{\thesection.}{0.5em}{}[]				%
\titlespacing*{\section}{0em}{2em}{1.5em}				%

\titleformat{\subsection}[runin]						%
{\bfseries\normalsize}{\thesubsection.}{0em}{\ }[.]		%

\theoremstyle{plain}								%
\newtheorem{thm}{Theorem}[section]					%
\newtheorem{prop}[thm]{Proposition}					%
\newtheorem{lem}[thm]{Lemma}						%
\newtheorem{cor}[thm]{Corollary}					%

\theoremstyle{definition}							%
\newtheorem{defn}[thm]{Definition}					%
\newtheorem{eg}[thm]{Example}						%

\theoremstyle{remark}								%
\newtheorem{rem}[thm]{Remark}						%

\numberwithin{equation}{section}						%

\DeclareMathOperator*{\esslim}{esslim}				
\DeclareMathOperator*{\esssup}{esssup}			
\DeclareMathOperator{\sgn}{sgn}					
\DeclareMathOperator{\supp}{supp}					

\newcommand{\Romannum}[1]{\uppercase\expandafter{\romannumeral#1\relax}}
\newcommand{\romannum}[1]{\romannumeral#1\relax}

\def \ve {\varepsilon}			
\def \La {\Lambda}				
\def \la {\lambda}				
\def \vf {\varphi}				

\def \bR {\mathbb R}			
\def \bZ {\mathbb Z}			

\def \cC {\mathcal C}			%
\def \CS {Cauchy--Schwarz }

\def \Gro {Gr\"onwall's }

\begin{document}

\pagestyle{plain}
\pagenumbering{arabic}
\bibliographystyle{plain}

\title{Scalar conservation law in a bounded domain with strong source at boundary}
\author{Lu \textsc{Xu}}
\date{}
\maketitle



\begin{abstract}
We consider a scalar conservation law with source in a bounded open interval $\Omega\subseteq\bR$.
The equation arises from the macroscopic evolution of an interacting particle system.
The source term models an external effort driving the solution to a given function $\varrho$ with an intensity function $V:\Omega\to\bR_+$ that grows to infinity at $\partial\Omega$.
We define the entropy solution $u \in L^\infty$ and prove the uniqueness.
When $V$ is integrable, $u$ satisfies the boundary conditions introduced in \cite{Otto96}, which allows the solution to attain values at $\partial\Omega$ different from the given boundary data.
When the integral of $V$ blows up, $u$ satisfies an energy estimate and presents essential continuity at $\partial\Omega$ in a weak sense.
\end{abstract}



\section{Introduction}

In this paper, we study the following initial-boundary value problem for a quasilinear scalar balance law in the bounded interval $(0,1)\subseteq\bR$ given by
\begin{equation}\label{eq:bl}
  \left\{
  \begin{aligned}
    &\partial_tu + \partial_x[J(u)] + G=0, \quad t>0,\ x\in(0,1),\\
    &u(0,\cdot)=u_0, \quad u(\cdot,0) = \alpha, \quad u(\cdot,1) = \beta.
  \end{aligned}
  \right.
\end{equation}
where the source term $G=G(t,x,u)$ reads
\begin{align}\label{eq:source}
  G(t,x,u)=V(x)(u-\varrho(t,x)),
\end{align}
and $J$, $V$, $\varrho$ are nice functions defined respectively on $\bR$, $(0,1)$ and $\bR_+\times(0,1)$.
Since the weak solution to \eqref{eq:bl} is not unique, we need to consider the \emph{entropy solution} obtained through the \emph{vanishing viscosity limit}.
The entropy solution presents discontinuities both inside $(0,1)$ and at the boundaries.
In particular, the values of $u$ at $\{0,1\}$ can be different from the prescribed boundary data $(\alpha,\beta)$, so the boundary conditions are a priori formal.
The first definition of the entropy solution is given in \cite{BLN79} for smooth $u_0$ and homogeneous boundary $(\alpha,\beta)\equiv(0,0)$.
It is then generalized in \cite{Otto96,Martin07} to the case with $u_0$, $\alpha$ and $\beta$ being $L^\infty$ functions, see also \cite[Section 2.6]{MNRR96}.
These definitions provide a set of possible boundary values, reflecting the formulation of \emph{boundary layer} during the vanishing viscosity limit.
We refer to \cite{ColomR15,DeFilippG17,Rossi19,BressanS20} and \cite[Section 6.9]{Daferm16} and references therein for more details and recent development.

Suppose that $V(x)>0$, then $G=G(t,x,u)$ satisfies that $\partial_uG>0$ and $G(\cdot,\cdot,\varrho)\equiv0$, i.e., $G$ acts as a source (resp. sink) when $u$ is less (resp. greater) than $\varrho$.
When $\varrho$ is a constant, \eqref{eq:bl}--\eqref{eq:source} can be viewed as a conservation system with \emph{relaxation} introduced in \cite{Liu87}, with the first component degenerated to a stationary solution.
In this paper, we aim at understanding the effect on the boundary discontinuities caused by \emph{extremely strong perturbation}.
Roughly speaking, suppose that $V\to\infty$ as $x\to0$, $1$ and choose $\varrho$ that is compatible to the boundary data: $\varrho|_{x=0}=\alpha$, $\varrho|_{x=1}=\beta$.
We define the $L^\infty$ entropy solution and prove the well-posedness.
We then investigate its behavior near the boundaries and show that the appearance of discontinuity is dependent on the \emph{integrability} of $V$.
Generally speaking,
\begin{itemize}
\item If $V$ is integrable, the boundary condition provides a set of possible values for $u$ at $x=0$ (resp. $x=1$) which can be different from $\alpha$ (resp. $\beta$).
The compatibility conditions are not necessary here.
\item If the integral of $V$ is divergent at $x\in\{0,1\}$, $u$ satisfies an energy estimate which prescribes the boundary values in a weak sense, and one always observes continuous flux at the boundaries.
\end{itemize}

\subsection{Physical motivation}

The equation studied in this paper arises naturally from the \emph{hydrodynamic limit} for asymmetric exclusion process with open boundaries \cite{Baha12,Xu22,Xu24b,XuZ24}.
It is an open interacting particle system that describes the dynamics of stochastic lattice gas with hard core repulsion.
Observed at properly chosen macroscopic space-time scale, the particle density evolves with a balance law with boundary conditions.

Consider the one-dimensional finite lattice $\La_N=\{1,\ldots,N-1\}$.
A variable $\eta_i$ is assigned to each site $i \in \La_N$, with $\eta_i = 0$ if the site is empty and $\eta_i=1$ if it is occupied by a particle.
The configuration is denoted by
\begin{align}
  \eta=(\eta_1,\ldots,\eta_{N-1}) \in \{0,1\}^{\La_N}.
\end{align}
The dynamics is described as following.
If there is a particle at site $i$, it waits for a random time $\tau$ distributed as $P(\tau>t)=e^{-t}$ and jumps to another \emph{vacant} site $i'>i$ on its right with probability $p_\gamma(i'-i)$, where
\begin{align}
  p_\gamma(k) := \frac{c_\gamma\mathbf1_{k>0}}{k^{1+\gamma}}, \quad c_\gamma^{-1} = \sum_{k=1}^\infty \frac1{k^{1+\gamma}},
\end{align}
and $\gamma>1$ is a constant.
We assume that the waiting times for all particles and all jumps are independent.

To model the boundary effects, we attach the system with two \emph{infinitely extended reservoirs}.
Suppose that one box containing infinitely many particles is placed at each site $j\in\bZ$, $j \le 0$.
The particles can enter and exit $\La_N$ obeying the following rules.
Particles in the box $j<0$ can jump to any empty site $i\in\La_N$ with rate $\alpha p_\gamma(|i-j|)$, and particle at site $i\in\La_N$ can jump back to the box $j<0$ with rate $(1-\alpha)p_\gamma(|i-j|)$.
Here, $\alpha\in(0,1)$ is a given deterministic number that stands for the density of the reservoirs.
Similar reservoirs with density $\beta\in(0,1)$ are placed at sites $j\in\bZ$, $j \ge N$.

Let $L_{\mathrm{exc},N}$, $L_{-,N}$ and $L_{+,N}$ be the infinitesimal generators of the exclusion dynamics, left and right reservoirs, respectively.
For $f:\{0,1\}^{\La_N}\to\bR$, they are precisely given by
\begin{equation}
  \begin{aligned}
    L_{\mathrm{exc},N}f(\eta) &= \sum_{i,i'\in\La_N} c(i,i',\eta)\big[f(\eta^{i,i'})-f(\eta)\big],\\
    L_{-,N}f(\eta) &= \sum_{j \le 0}\sum_{i\in\La_N} c_-(i,j,\eta)\big[f(\eta^i)-f(\eta)\big],\\
    L_{+,N}f(\eta) &= \sum_{j \ge N}\sum_{i\in\La_N} c_+(i,j,\eta)\big[f(\eta^i)-f(\eta)\big],
  \end{aligned}
\end{equation}
where $\eta^{i,i'}$ is the configuration obtained by exchanging $\eta_i$ and $\eta_{i'}$ in $\eta$, $\eta^i$ is the one obtained by flipping $\eta_i$ to $1-\eta_i$ in $\eta$, and
\begin{equation}
  \begin{aligned}
    c(i,i',\eta) &= p_\gamma(i'-i)\eta_i(1-\eta_{i'}),\\
    c_-(i,j,\eta) &= \alpha p_\gamma(|i-j|)(1-\eta_i) + (1-\alpha)p_\gamma(|i-j|)\eta_i,\\
    c_+(i,j,\eta) &= \beta p_\gamma(|i-j|)(1-\eta_i) + (1-\beta)p_\gamma(|i-j|)\eta_i.
  \end{aligned}
\end{equation}

Consider the Markov process $\{\eta(t)=\eta^N(t);t\ge0\}$ generated by
\begin{align}
  L_N = NL_{\mathrm{exc},N} + N^\gamma\big(L_{-,N} + L_{+,N}\big).
\end{align}
The factor $N$ means that the dynamics of exclusion on $\La_N$ is accelerated to the hyperbolic scale $Nt$.
Meanwhile, $N^\gamma$ corresponds to a different scale for the reservoirs, for which the reason will be clarified later.

Assume some $u_0 \in L^\infty((0,1))$, such that
\begin{align}\label{eq:hd-initial}
  u_0^N(x) := \sum_{i=1}^{N-1} \eta_i^N(0)\chi_{[\frac iN-\frac1{2N},\frac iN+\frac1{2N})}(x) \overset{N\to\infty}{\longrightarrow} u_0(x)
\end{align}
in probability, which precisely means that
\begin{align}
  \lim_{N\to\infty} P \left\{ \left| \int_0^1 u_0^N(x)g(x)dx - \int_0^1 u_0(x)g(x)dx \right| > \delta \right\} = 0
\end{align}
for any $\delta>0$ and continuous function $g$.
The \emph{hydrodynamic limit} corresponds to the convergence that for almost every $t>0$,
\begin{align}
  u^N(t,x) := \sum_{i=1}^{N-1} \eta_i^N(t)\chi_{[\frac iN-\frac1{2N},\frac iN+\frac1{2N})}(x) \overset{N\to\infty}{\longrightarrow} u(t,x)
\end{align}
in probability.
Since $\gamma>1$, $p_\gamma$ possesses finite first moment: $\mathfrak p_\gamma := \sum_{k>0} kp_\gamma(k) <\infty$.
Hence, without considering the effects of reservoirs, $u$ is the entropy solution to (see \cite{SethurS18}):
\begin{align}
  \partial_tu + \mathfrak p_\gamma\partial_x[u(1-u)] = 0, \quad t>0, \ x\in(0,1).
\end{align}
To investigate the effect of the left reservoirs, observe that
\begin{equation}
  \begin{aligned}
    L_{-,N} [\eta_i] &= \sum_{j\le0} c_-(i,j,\eta)(1-2\eta_i)\\
    &= (\alpha-\eta_i)\sum_{k \ge i} p_\gamma(k) \approx (\alpha-\eta_i)\frac{c_\gamma}{i^\gamma\gamma }.
  \end{aligned}
\end{equation}
The factor $N^\gamma$ is chosen to get the non-trivial limit
\begin{equation}
  \begin{aligned}
    &N^\gamma L_{-,N} \left[ \frac1N\sum_{i=1}^{N-1} \eta_i(t)g \left( \frac iN \right) \right]\\
    \approx\,&\frac1N\sum_{i=1}^{N-1} (\alpha-\eta_i) \frac{c_\gamma N^\gamma}{i^\gamma\gamma}g \left( \frac iN \right) \to \frac{c_\gamma}\gamma \int_0^1 \frac{(\alpha-u)g}{x^\gamma}dx.
  \end{aligned}
\end{equation}
Similar argument works for the right reservoir.
Putting them together, we obtain formally the following hydrodynamic equation
\begin{align}\label{eq:hl}
  \partial_tu + \mathfrak p_\gamma\partial_x[u(1-u)] + \frac{c_\gamma}\gamma \left[ \frac{u-\alpha}{x^\gamma} + \frac{u-\beta}{(1-x)^\gamma} \right] = 0,
\end{align}
for $x\in(0,1)$, with the natural initial and boundary conditions
\begin{align}\label{eq:hl-bd}
  u(0,x)=u_0(x), \quad u(t,0)=\alpha, \quad u(t,1)=\beta.
\end{align}
The source term can be written as $V(x)(u-\varrho(x))$, where
\begin{align}\label{eq:example}
  V = \frac{c_\gamma}\gamma \left[ \frac1{x^\gamma} + \frac1{(1-x)^\gamma} \right], \quad \varrho = \frac{\alpha(1-x)^\gamma+\beta x^\gamma}{x^\gamma+(1-x)^\gamma}.
\end{align}
Conservation law with general $V$ and $\varrho$ can be modelled by exclusion process with \emph{Glauber dynamics}, see \cite{XuZ24} for details.

Note that in \eqref{eq:hd-initial}, the total variation of the initial empirical density $u_0^N$ can grow in order $\mathcal O(N)$.
For this reason, we focus on constructing the entropy solution in $L^\infty$ space, rather than in the space of bounded-variation functions.

\begin{rem}
Assume some $t_0>0$ such that \eqref{eq:hl} has a classical solution for $t<t_0$.
Using the method of characteristics, one obtains the characteristic equation associated to \eqref{eq:hl}:
\begin{align}
  x(0)=x_0\in(0,1), \quad x'(t) = \mathfrak p_\gamma\big[1-2u(t,x(t))\big].
\end{align}
Let $v(t):=u(t,x(t))$ for $t\in[0,t_0)$, then
\begin{align}\label{eq:characteristic}
  v(0) = u_0(x_0), \quad v'(t) = V(x(t))\big[\varrho(x(t))-v(t)\big],
\end{align}
where $V(x)$ and $\varrho(x)$ are functions given by \eqref{eq:example}.
Hence, we formally obtain the second-order ordinary differential equation for the characteristic:
\begin{equation}
  \left\{
  \begin{aligned}
    &\,x''(t) + V(x(t))x'(t) = \mathfrak p_\gamma V(x(t)) \big[1-2\varrho(x(t))\big],\\
    &\,x(0) = x_0, \quad x'(0) = \mathfrak p_\gamma(1-u_0(x_0)).
  \end{aligned}
  \right.
\end{equation}
The classical solution is then determined by \eqref{eq:characteristic} along these lines.
\end{rem}

\section{Model and main results}

Denote $\Sigma=\bR_+\times(0,1)$.
Through this paper, we consider the equation \eqref{eq:bl}--\eqref{eq:source} on $\Sigma$.
The following conditions are always assumed.
\begin{itemize}
\item[(h1)] $J\in\cC^1(\bR;\bR)$.
\item[(h2)] $V\in\cC((0,1);\bR_+))$ satisfies that
\begin{align}\label{eq:ass-v-0}
  \lim_{x\to0+} V(x) = \lim_{x\to1-} V(x) = \infty.
\end{align}
\item[(h3)] The initial data $u_0$, the boundary data $\alpha$, $\beta$ and $\varrho$ are measurable, essentially bounded functions on $(0,1)$, $\bR_+$ and $\Sigma$, respectively.
\end{itemize}

Our first aim is to define the unique entropy solution to \eqref{eq:bl}--\eqref{eq:source} in $L^\infty(\Sigma)$.
The concept of \emph{Lax entropy--flux pair} plays a central role.

\begin{defn}
A function $f\in\cC^2(\bR)$ is called a Lax entropy associated to \eqref{eq:bl} and $q\in\cC^2(\bR)$ is called the corresponding flux, if
\begin{align}
  f''(u)\ge0, \quad q'(u)=f'(u)J'(u), \quad \forall\,u\in\bR.
\end{align}
\end{defn}

As mentioned before, the properties of the entropy solution rely heavily on the integrability of $V$.
Hereafter, we distinguish two cases.

\subsection{Integrable case}

The source $G$ is called integrable when $V$ belongs to $L^1((0,1))$.
In this case, we begin with Otto's definition of boundary entropy and the corresponding flux \cite{Otto96}.

\begin{defn}
\label{defn:bd-ent}
$(F,Q)\in\cC^2(\bR^2;\bR^2)$ is called a \emph{boundary entropy--flux pair}, if the next two conditions are satisfied.
\begin{enumerate}
\item $(f,q) := (F,Q)(\cdot,k)$ is a Lax entropy--flux pair for all $k\in\bR$,
\item $F(k,k) = \partial_uF(u,k)|_{u=k} = Q(k,k) = 0$ for all $k\in\bR$.
\end{enumerate}
\end{defn}

The definition of entropy solution to \eqref{eq:bl} for the integrable case is similar to the case without $V$ (see, e.g., \cite[Definition 2.7.2, Theorem 2.7.31]{MNRR96}) or with bounded $V$ (see, e.g., \cite[Definition 2.1]{ColomR15}).

\begin{defn}
\label{defn:in}
Assume $V \in L^1((0,1))$.
The entropy solution to \eqref{eq:bl} is a function $u \in L^\infty(\Sigma)$ that satisfies the generalized entropy inequality
\begin{equation}\label{eq:ei-bd}
  \begin{aligned}
    \int_0^1 F(u_0,k)\vf(0,\cdot)dx + \iint_\Sigma \big[F(u,k)\partial_t\vf + Q(u,k)\partial_x\vf\big]dxdt\\
    \ge \iint_\Sigma \partial_uF(u,k)V(x)(u-\varrho)\vf\,dxdt \hspace{20mm}\\
    -\,M\int_0^T \big[F(\alpha,k)\vf(\cdot,0)+F(\beta,k)\vf(\cdot,1)\big]dt,
  \end{aligned}
\end{equation}
for all boundary entropy--flux pairs $(F,Q)$, $k\in\bR$, and $\vf\in\cC_c^2(\bR^2)$ such that $\vf\ge0$.
In \eqref{eq:ei-bd}, the constant $M$ is given by
\begin{align}
  M:=\sup \Big\{|J'(u)|; |u|\le \esssup\big\{|\varrho|,|\alpha|,|\beta|,|u_0|\big\}\Big\}.
\end{align}
\end{defn}

As a standard result, the smooth entropy--flux pairs in Definition \ref{defn:in} can be replaced by non-smooth ones, and the initial condition holds in $L^1$.

\begin{defn}
For $(u,k) \in \bR^2$, define
\begin{align}
  \eta(u,k):=|u-k|, \quad \xi(u,k):=\sgn(u-k)[J(u)-J(k)].
\end{align}
The pair $(\eta,\xi)$ is called the \emph{Kruzhkov entropy--flux pair}.
\end{defn}

\begin{prop}\label{prop:kuz-in}
Assume $V \in L^1((0,1))$.
The entropy solution is equivalently defined as $u \in L^\infty(\Sigma)$ such that
\begin{equation}\label{eq:ei-bd-kuz}
  \begin{aligned}
    \int_0^1 |u_0-k|\vf(0,\cdot)dx + \iint_\Sigma \big[|u-k|\partial_t\vf + 
  \xi(u,k)\partial_x\vf\big]dxdt\\
    \ge \iint_\Sigma \sgn(u-k)V(x)(u-\varrho)\vf\,dxdt \hspace{20mm}\\
    -\,M\int_0^T \big[|\alpha-k|\vf(\cdot,0)+|\beta-k|\vf(\cdot,1)\big]dt,
  \end{aligned}
\end{equation}
for all $k\in\bR$ and $\vf\in\cC_c^2(\bR^2)$ such that $\vf\ge0$.
Moreover,
\begin{align}\label{eq:ini}
  \esslim_{t\to0+} \int_0^1 |u(t,x)-u_0(x)|dx = 0.
\end{align}
\end{prop}

Using the methods in \cite[Section 2.7 \& 2.8]{MNRR96}, we obtain the well-posedness of $u$ and an explicit expression for the boundary conditions.

\begin{prop}
Assume that $V \in L^1((0,1))$, then \eqref{eq:bl} admits a unique entropy solution $u \in L^\infty(\Sigma)$.
\end{prop}

\begin{prop}
Let $u$ be as in Definition \ref{defn:in}.
For all $0<s<t$ and boundary entropy--flux pairs $(F,Q)$,
\begin{equation}\label{eq:bd-in}
  \begin{aligned}
    \esslim_{x\to0+} \int_s^t Q\big(u(r,x),\alpha(r)\big)dr \le 0,\\
    \esslim_{x\to1-} \int_s^t Q\big(u(r,x),\beta(r)\big)dr \ge 0.
  \end{aligned}
\end{equation}
\end{prop}

\subsection{Non-integrable case}

The source $G$ is called non-integrable when the integral of $V$ is infinite.
In this case, the singular points of the integral of $V$ can only be $\{0,1\}$.
We will see later in Remark \ref{rem:mix-bd} that, when the integral of $V$ is divergent at only one of them, the equation can be treated as a mixed boundary problem with one side integrable and the other side non-integrable.
Hence, we assume without loss of generality that
\begin{align}\label{eq:ass-v}
  \int_0^y V(x)dx = \int_{1-y}^1 V(x)dx = \infty, \quad \forall\,y\in(0,1).
\end{align}
Also assume the compatibility conditions: for all $T>0$
\begin{equation}\label{eq:ass-rho-l2}
  \begin{aligned}
    &\lim_{y\to0+} \int_0^T \int_0^y V(x)\big[\varrho(t,x)-\alpha(t)\big]^2dxdt = 0,\\
    &\lim_{y\to0+} \int_0^T \int_{1-y}^1 V(x)\big[\varrho(t,x)-\beta(t)\big]^2dxdt = 0.
  \end{aligned}
\end{equation}
Notice that \eqref{eq:ass-rho-l2} is generally true in the integrable case, since $\varrho$, $\alpha$ and $\beta$ are essentially bounded.

\begin{defn}
\label{defn:non}
Assume \eqref{eq:ass-v} and \eqref{eq:ass-rho-l2}.
The entropy solution to \eqref{eq:bl} is a function $u \in L^\infty(\Sigma)$ that satisfies the following conditions.
\begin{enumerate}
\item[(EB)] The energy bound: for all $T>0$,
\begin{align}\label{eq:eb}
  \int_0^T \int_0^1 V(x)\big[u(t,x)-\varrho(t,x)\big]^2dxdt < \infty.
\end{align}
\item[(EI)] The generalized entropy inequality
\begin{equation}\label{eq:ei}
  \begin{aligned}
    \int_0^1 f(u_0)\vf(0,\cdot)dx + \iint_\Sigma \big[f(u)\partial_t\vf + q(u)\partial_x\vf\big]dxdt\\
    \ge \iint_\Sigma f'(u)V(x)(u-\varrho)\vf\,dxdt,
  \end{aligned}
\end{equation}
for all Lax entropy--flux pairs $(f,q)$ and all $\vf\in\cC_c^2(\bR\times(0,1))$ such that $\vf\ge0$,
\end{enumerate}
\end{defn}

\begin{rem}
Despite that \eqref{eq:ei} contains no boundary condition, it turns out that the entropy solution is unique, see Theorem \ref{thm:uniq} and \ref{thm:bd} below.
Indeed, from \eqref{eq:ass-rho-l2} and \eqref{eq:eb},
\begin{equation}\label{eq:bd-l2}
  \begin{aligned}
    &\lim_{y\to0+} \int_0^T \int_0^y V(x)\big[u(t,x)-\alpha(t)\big]^2dxdt = 0,\\
    &\lim_{y\to0+} \int_0^T \int_{1-y}^1 V(x)\big[u(t,x)-\beta(t)\big]^2dxdt = 0.
  \end{aligned}
\end{equation}
Given \eqref{eq:ass-v}, the necessary boundary information is contained here.
More details can be found in \eqref{eq:bd-average} and Lemma \ref{lem:bd}.
\end{rem}

\begin{rem}\label{rem:mix-bd}
Indeed, the boundary conditions at $x=0$ and $x=1$ are treated separately.
Hence, if $V$ is integrable at $x=0$ (resp. $x=1$) but not at $x=1$ (resp. $x=0$), the entropy solution is defined by \eqref{eq:eb} and \eqref{eq:ei-bd} for all $\vf\in\cC_c^2(\bR\times(-\infty,1))$ (resp. $\cC_c^2(\bR\times(0,\infty))$) such that $\vf\ge0$.
\end{rem}

Similarly to the integrable case, we can define the entropy solution using the Kruzhkov entropy instead.

\begin{prop}\label{prop:kuz-non}
Assume \eqref{eq:ass-v} and \eqref{eq:ass-rho-l2}.
The entropy solution is equivalently defined as $u \in L^\infty(\Sigma)$ satisfying \textnormal{(EB)} and
\begin{equation}\label{eq:ei-kuz}
  \begin{aligned}
    \int_0^1 |u_0-k|\vf(0,\cdot)dx + \iint_\Sigma \big[|u-k|\partial_t\vf + \xi(u,k)\partial_x\vf\big]dxdt\\
    \ge \iint_\Sigma \sgn(u-k)V(x)(u-\varrho)\vf\,dxdt,
  \end{aligned}
\end{equation}
for all $k\in\bR$ and $\vf\in\cC_c^2(\bR\times(0,1))$ such that $\vf\ge0$.
Furthermore, the initial data is attained in the sense of \eqref{eq:ini}.
\end{prop}

We are now ready to state our main results.

\begin{thm}[Uniqueness]\label{thm:uniq}
Assume \eqref{eq:ass-rho-l2}.
Instead of \eqref{eq:ass-v}, assume that $V$ satisfies a stronger condition at the boundaries:
\begin{align}\label{eq:ass-v-uniq}
  \limsup_{y\to0+}\,\frac1{y^2}\int_0^y \left[ \frac1{V(x)} + \frac1{V(1-x)}\right] dx < \infty.
\end{align}
Then, there is at most one $u \in L^\infty(\Sigma)$ that satisfies Definition \ref{defn:non}.
\end{thm}

Observe that for any $\delta>0$ and $y\in(0,1)$,
\begin{equation}
  \begin{aligned}
    &\int_0^T \frac1y\int_0^y |u(t,x)-\alpha(t)|\,dxdt\\
    \le\,&\frac1{4\delta}\int_0^T \int_0^y V(x)(u-\alpha)^2dxdt + \frac{T\delta}{y^2}\int_0^y \frac1{V(x)}dx.
  \end{aligned}
\end{equation}
Taking $y\to0$ and choosing $\delta$ arbitrarily small, \eqref{eq:ass-v-uniq} suggests that the boundary conditions in \eqref{eq:bl} hold in the sense of space-time average:
\begin{align}\label{eq:bd-average}
  \lim_{y\to0} \int_0^T \frac1y\int_0^y |u(t,x)-\alpha(t)|\,dxdt = 0,
\end{align}
and similarly for $\beta(t)$.
The convergence in \eqref{eq:bd-average} can be significantly improved under extra conditions.

\begin{thm}\label{thm:bd}
Let $\alpha(t)\equiv\alpha$, $\beta(t)\equiv\beta$ be almost everywhere constants and $u$ satisfy \textnormal{(EB)} and \textnormal{(EI)}.
Assume \eqref{eq:ass-v-uniq} and for all $T>0$ that
\begin{equation}\label{eq:ass-rho-l1}
  \begin{aligned}
    &\lim_{y\to0+} \int_0^T \int_0^y V(x)|\varrho(t,x)-\alpha|\,dxdt = 0,\\
    &\lim_{y\to0+} \int_0^T \int_{1-y}^1 V(x)|\varrho(t,x)-\beta|\,dxdt = 0.
  \end{aligned}
\end{equation}
Then, for all $0 \le s < t$ and Lax entropy--flux pairs $(f,q)$,
\begin{equation}\label{eq:bd}
  \begin{aligned}
    \esslim_{x\to0+} \int_s^t q(u(r,x))dr = (t-s)q(\alpha),\\
    \esslim_{x\to1-} \int_s^t q(u(r,x))dr = (t-s)q(\beta).
  \end{aligned}
\end{equation}
\end{thm}

\begin{cor}\label{cor:bd}
Assume the same conditions as in Theorem \ref{thm:bd} and that $J$ is convex (or concave), then for all $0 \le s < t$,
\begin{equation}\label{eq:bd-1}
  \begin{aligned}
    \esslim_{x\to0+} \int_s^t u(r,x)dr = (t-s)\alpha,\\
    \esslim_{x\to1-} \int_s^t u(r,x)dr = (t-s)\beta.
  \end{aligned}
\end{equation}
\end{cor}

Finally, the existence of the entropy solution for non-integrable source with smooth coefficients and boundary data is established below.

\begin{thm}[Existence]\label{thm:exis}
Assume that $J\in\cC^2(\bR)$, $V\in\cC^2((0,1))$, $\alpha$, $\beta\in\cC_b^2(\bR_+)$ satisfy \eqref{eq:ass-v} and \eqref{eq:ass-rho-l2}.
Moreover, suppose that for each $T>0$, there is a family of functions $\{\varrho^\ve;\ve>0\}$ such that the following conditions are fulfilled.
\begin{itemize}
\item[\textnormal{(\romannum1)}] For each $\ve>0$, $\varrho^\ve\in\cC^2(\Sigma_T)$ and $\varrho^\ve\to\varrho$ in $L^2(\Sigma_T)$.
\item[\textnormal{(\romannum2)}] $\|\varrho^\ve\|_{L^\infty(\Sigma_T)} \le \|\varrho\|_{L^\infty(\Sigma_T)}$, $\sup_{\ve>0} \|\varrho^\ve\|_{H^1(\Sigma_T)} < \infty$ and
\begin{align}
  \sup_{\ve>0} \iint_{\Sigma_T} V(x)\big[\varrho^\ve(t,x)-\varrho(t,x)\big]^2dxdt < \infty.
\end{align}
\end{itemize}
Then, \eqref{eq:bl} admits an entropy solution in Definition \ref{defn:non}.
\end{thm}
\begin{eg}
Recall \eqref{eq:hl} with boundary conditions \eqref{eq:hl-bd}.
When $\gamma\in(0,1)$, the source is integrable.
When $\gamma\ge1$, the source is non-integrable and the conditions in Theorem \ref{thm:uniq} and \ref{thm:exis} are satisfied.
Hence, the particle density evolves macroscopically with the unique entropy solution.
\end{eg}

\begin{rem}
The method presented for integrable $V$ can be extended to scalar balance laws in spatial dimensions $d\ge2$, see, e.g., \cite{Otto96,Martin07}.
For the multi-dimensional non-integrable case, one can construct an entropy solution satisfying an energy estimate similar to \eqref{eq:eb} via the standard vanishing viscosity limit.
However, the corresponding uniqueness remains open.
\end{rem}

\subsection{Organization of the paper}

The arguments for the integrable case are largely the same as those used in \cite[Section 2.7, 2.8]{MNRR96}, see also \cite{ColomR15}.
Hence, we only summarize the ideas briefly.
The focus is the non-integrable case.
In Section \ref{sec:uniq}, we prove Theorem \ref{thm:uniq} exploiting Kruzhkov's doubling of variables technique.
In Section \ref{sec:bd}, we prove Theorem \ref{thm:bd} via an $L^1$-refinement of the energy bound.
In Section \ref{sec:exis}, we prove Theorem \ref{thm:exis} with vanishing viscosity method.
Proposition \ref{prop:kuz-non} and some preliminary results are proved in Section \ref{sec:pre}.

\subsection{Notations}
For a measure space $(X;\mu)$ and $p\ge1$, let
\begin{align}
  L^p(X;\mu) = \big\{f;\|f\|_{L^p(X;\mu)}<\infty\big\}, \quad \|f\|_{L^p(X;\mu)}^p = \int_X |f|^p\,d\mu.
\end{align}
For $p=\infty$, $L^\infty(X;\mu)$ stands for the space of essentially bounded measurable functions and $\|\cdot\|_{L^\infty(X;\mu)}$ is the essential supremum norm.
When $X\subseteq\bR^d$ and $\mu$ is the Lebesgue measure, we use the abbreviations $L^p(X)$ and $\|\cdot\|_{L^p(X)}$.

Recall that $\Sigma=\bR_+\times(0,1)$ and denote by $\nu$ the $\sigma$-finite measure on $\Sigma$ given by $\nu(dxdt)=V(x)dxdt$.
For $T>0$, let $\Sigma_T=(0,T)\times(0,1)$.
With some abuse of notations, the restriction of $\nu$ on $\Sigma_T$ is still denoted by $\nu$.

Let $(f,q)$ be either a Lax entropy--flux pair or the Kruzhkov entropy--flux pair $(\eta,\xi)(\cdot,k)$.
For $\vf\in\cC_c^2(\bR^2)$, the entropy product of $(f,q)$ is defined as
\begin{equation}\label{eq:ent-prod}
  \begin{aligned}
    E_\vf^{(f,q)}(u) := \iint_\Sigma \big[f(u)\partial_t\vf+q(u)\partial_x\vf\big]dxdt\\
    - \iint_\Sigma f'(u)V(x)(u-\varrho)\vf\,dxdt.
  \end{aligned}
\end{equation}
We identify $\partial_u\eta(u,k) = \sgn(u-k)$ for $\eta=|u-k|$.
Notice that the last integral in \eqref{eq:ent-prod} is well-defined if and only if $f'(u)(u-\varrho)\vf \in L^1(\Sigma;\nu)$.

\section{Preliminary results}
\label{sec:pre}

First, we verify the alternative definitions of the entropy solution with \eqref{eq:ei-bd-kuz} and \eqref{eq:ei-kuz}.
The integrability of $V$ is irrelevant here.

\begin{lem}
For $u \in L^\infty(\Sigma)$, \eqref{eq:ei-bd} holds for all Lax entropy--flux pairs if and only it holds for the Kruzhkov entropy--flux pair and all $k\in\bR$.
\end{lem}

\begin{proof}
Choose $g\in\cC^2(\bR)$ such that $g(0)=g'(0)=0$, $g(2)=g'(2)=1$, $g''(u)\ge0$ and $g(u)=g(-u)$.
For $\ve>0$, define
\begin{equation}\label{eq:kuz-appro}
  \begin{aligned}
    F_\ve(u,k) &:=
    \begin{cases}
      \,|u-k|-\ve, &|u-k| > 2\ve,\\
      \,\ve g(\ve^{-1}(u-k)), &|u-k| \le 2\ve,
    \end{cases}
    \\
    Q_\ve(u,k) &:= \int_k^u \partial_uF_\ve(w,k)J'(w)dw.
  \end{aligned}
\end{equation}
Observe that $(F_\ve,Q_\ve)(\cdot,k)$ is a Lax entropy--flux pair for each $\ve$, $(F_\ve,Q_\ve)(\cdot,k) \rightrightarrows (\eta,\xi)(\cdot,k)$ as $\ve\to0$, and
\begin{align}
  \partial_uF_\ve(u,k) =
  \begin{cases}
    \sgn(u-k), &|u-k|>2\ve,\\
    g'(\ve^{-1}(u-k)), &|u-k|\le2\ve.
  \end{cases}
\end{align}

Suppose that \eqref{eq:ei-bd} holds for all Lax entropy--flux pairs $(f,q)$.
To get \eqref{eq:ei-bd-kuz}, it suffices to take $(F_\ve,Q_\ve)(\cdot,k)$ in \eqref{eq:ei-bd} and let $\ve\to0$.
On the other hand, assume \eqref{eq:ei-bd-kuz} for all $k\in\bR$.
Since $k$ can be chosen smaller than $-\|u\|_{L^\infty(\Sigma)}$, \eqref{eq:ei-bd} is true for the linear entropy $f = u$ and the corresponding flux $q = J$.
Then, one only needs to use the fact that any Lax entropy--flux pair $(f,q)$ is contained in the convex hull of $(\eta,\xi)(\cdot,k)$ and $(u,J)$.
\end{proof}

The proofs of Proposition \ref{prop:kuz-in} and \ref{prop:kuz-non} are standard.
Below we assume the non-integrable case and prove Proposition \ref{prop:kuz-non} as an example.

\begin{proof}
The first argument follows directly from the previous lemma.
To verify the $L^1$-continuity at $t=0$, we use the idea in \cite[Lemma 2.7.34, 2.7.41]{MNRR96}.
For any $\phi\in\cC_c^2(\bR)$ and $\psi\in\cC_c^2((0,1))$ such that $\phi$, $\psi\ge0$, let $\vf=\phi(t)\psi(x)$.
The entropy product in \eqref{eq:ent-prod} reads
\begin{equation}
  \begin{aligned}
    E_\vf^{(\eta,\xi)(\cdot,k)}(u) &=\iint_\Sigma |u-k|\phi'\psi\,dxdt\\
    &+ \iint_\Sigma \big[\xi(u,k)\psi' - \sgn(u-k)V(x)(u-\varrho)\psi\big]\phi\,dxdt.
  \end{aligned}
\end{equation}
Recall that $u_0 \in L^\infty((0,1))$ and $u \in L^\infty(\Sigma)$.
Let $M=\|u_0\|_{L^\infty}+\|u\|_{L^\infty}$ and for $k\in[-M,M]$, $|\xi(u,k)| \le 2\sup_{[-M,M]} |J|$.
Hence, the second line above is bounded by
\begin{align}
  \left[ C_M\sup |\psi'| + \big(\|u\|_{L^\infty}+\|\varrho\|_{L^\infty}\big)\int_0^1 V(x)\psi(x)dx \right] \int_0^\infty \phi(t)dt.
\end{align}
Since $\psi$ is compactly supported within $(0,1)$,
\begin{align}
  E_\vf^{(\eta,\xi)(\cdot,k)}(u) \le \iint_\Sigma |u-k|\phi'\psi\,dxdt + C\int_0^\infty \phi(t)dt.
\end{align}
The generalized entropy inequality \eqref{eq:ei-kuz} then yields that
\begin{align}\label{eq:initial-1}
  \phi(0)\int_0^1 |u_0(x)-k|\psi(x)dx + \int_0^\infty F_{k,\psi}(t)\phi'(t)dt \ge 0,
\end{align}
where the function $F_{k,\psi}:(0,\infty)\to\bR$ is defined as
\begin{align}
  F_{k,\psi}(t) := \int_0^1 |u(t,x)-k|\psi(x)dx-Ct.
\end{align}
From \eqref{eq:initial-1}, after a possible modification on a set of zero measure, $F_{k,\psi}$ is non-increasing on $(0,\infty)$, and
\begin{align}
  \esslim_{t\to0+} F_{k,\psi}(t) \le \int_0^1 |u_0(x)-k|\psi(x)dx.
\end{align}
In other words, for all $\psi\in\cC_c^2((0,1))$ such that $\psi\ge0$,
\begin{align}\label{eq:initial-2}
  \esslim_{t\to0+} \int_0^1 |u(t,\cdot)-k|\psi\,dx \le \int_0^1 |u_0-k|\psi\,dx.
\end{align}
By a standard density argument, \eqref{eq:initial-2} holds for $\psi \in L^1((0,1))$ such that $\psi\ge0$.
One can approximate $v \in L^\infty((0,1))$ by simple functions taking only rational values to get
\begin{align}
  \esslim_{t\to0+} \int_0^1 |u(t,\cdot)-v|\psi\,dx \le \int_0^1 |u_0-v|\psi\,dx.
\end{align}
The result then follows by simply taking $v=u_0$ and $\psi\equiv1$.
\end{proof}

Next, we focus on the boundaries in the non-integrable case.
Pick a function $\psi\in\cC^\infty(\bR)$ such that
\begin{align}
  \supp \psi \in (0,\infty), \quad \psi|_{x\ge1} \equiv 1.
\end{align}
For $\ve>0$ and $x\in[0,1]$, define
\begin{align}\label{eq:aux-func}
  \psi_\ve(x) := \psi(\tfrac x\ve)\mathbf1_{\{x<\ve\}} + \mathbf1_{\{\ve \le x \le 1-\ve\}} + \psi(\tfrac{1-x}\ve)\mathbf1_{\{x>1-\ve\}}.
\end{align}
Then, $\psi_\ve\in\cC_c^\infty((0,1))$ and $\psi_\ve\to\mathbf1_{(0,1)}$ in $L^1((0,1))$ as $\ve\to0$.

\begin{lem}\label{lem:bd}
Suppose that \eqref{eq:ass-v-uniq} holds and $u \in L^\infty(\Sigma)$ satisfies \eqref{eq:bd-l2}.
Fix some $T>0$ and recall that $\Sigma_T=(0,T)\times(0,1)$.
Let $g$ be a measurable function on $\bR^3$ such that
\begin{align}
  \big|g(t,x,w)-g(t,x',w')\big| \le C\big(|x-x'|+|w-w'|\big)
\end{align}
for all $(t,x)$, $(t,x')\in\Sigma_T$ and $|w|$, $|w'| \le \|u\|_{L^\infty(\Sigma_T)}$.
Then,
\begin{align}
  \lim_{\ve\to0} \iint_{\Sigma_T} g(t,x,u)\psi'_\ve\,dxdt = \int_0^T \big[g(\cdot,0,\alpha(\cdot))-g(\cdot,1,\beta(\cdot))\big]dt.
\end{align}
\end{lem}

\begin{proof}
From the definition of $\psi_\ve$,
\begin{align}
  \iint_{\Sigma_T} g(\cdot,\cdot,u)\psi'_\ve\,dxdt = \int_0^T \left( \int_0^\ve+\int_{1-\ve}^1 \right) g(\cdot,\cdot,u)\psi'_\ve\,dxdt.
\end{align}
Noting that the integral of $\psi'_\ve(x)$ from $0$ to $\ve$ is $1$,
\begin{equation}
  \begin{aligned}
    &\left| \int_0^T \int_0^\ve g(t,x,u(t,x))\psi'_\ve(x)dxdt - \int_0^T g(t,0,\alpha(t))dt \right|\\
    \le\,&\int_0^T \int_0^\ve \big|g(t,x,u(t,x))-g(t,0,\alpha(t))\big|\psi'_\ve(x)dxdt
  \end{aligned}
\end{equation}
The condition of $g$ together with the fact that $|\psi'_\ve| \le C\ve^{-1}$ yields that the last line is bounded from above by
\begin{align}
  \frac C\ve\int_0^T \int_0^\ve \big(x+|u(t,x)-\alpha(t)|\big)dxdt.
\end{align}
Applying \CS inequality, we obtain the upper bound
\begin{equation}
  \frac{CT\ve}2 + \frac1\delta\int_0^T \int_0^\ve V(x)(u-\alpha)^2dxdt + \frac{C^2T\delta}{4\ve^2}\int_0^\ve \frac1{V(x)}dx,
\end{equation}
for any $\delta>0$.
Taking first $\ve\to0$ and then $\delta$ sufficiently small, we have
\begin{align}
  \lim_{\ve\to0} \int_0^T \int_0^\ve g(t,x,u(t,x))\psi'_\ve\,dxdt = \int_0^T g(t,0,\alpha(t))dt.
\end{align}
The integral over $(1-\ve,1)$ can be treated similarly.
\end{proof}

\section{Uniqueness of the entropy solution}
\label{sec:uniq}

In this section, we first prove the uniqueness of the entropy solution in the non-integrable case, then briefly summarize the difference when $V$ is integrable.
In the non-integrable case, the uniqueness is a direct consequence of the stability below.

\begin{thm}\label{thm:uniq-non}
Assume \eqref{eq:ass-rho-l2}, \eqref{eq:ass-v-uniq} and let $u$, $v$ satisfy Definition \ref{defn:non} with the same $(\alpha,\beta)$ and different $(u_0,\varrho)$, $(v_0,\varrho_*)$, respectively.
Then, for almost all $t>0$,
\begin{align}
  \int_0^1 |u(t,\cdot)-v(t,\cdot)|dx \le \int_0^1 |u_0-v_0|dx + \iint_{\Sigma_t} V(x)|\varrho-\varrho_*|\,dxds.
\end{align}
\end{thm}

The next Kruzhkov-type lemma plays a key role in the proof.
Since we choose the test function $\vf$ to be compactly supported in $\Sigma$, the integrability of $V$ is indeed irrelevant to either the statement or the proof.

\begin{lem}\label{lem:kuz}
For all $\vf\in\cC_c^2(\Sigma)$ such that $\vf\ge0$,
\begin{align}\label{eq:kuz}
  \iint_\Sigma \big[|u-v|\partial_t\vf + \xi(u,v)\partial_x\vf + |\varrho-\varrho_*|V\vf\big]dxdt \ge 0.
\end{align}
\end{lem}

\begin{proof}
Let $T>0$ be fixed and we verify \eqref{eq:kuz} for $\vf\in\cC_c^2(\Sigma_T)$.
Let $\phi\in\cC_c^\infty(\bR)$ be a mollifier such that
\begin{align}
  \supp \phi \subseteq (-1,1), \quad \phi(-\tau)=\phi(\tau), \quad \int_\bR \phi(\tau)d\tau=1.
\end{align}
For $\ve>0$, let $\phi_\ve(\tau,\zeta)=\ve^{-2}\phi(\ve^{-1}\tau)\phi(\ve^{-1}\zeta)$ and define
\begin{align}
  \Phi_\ve(t,x,s,y):=\vf \left( \frac{t+s}2,\frac{x+y}2 \right) \phi_\ve \left( \frac{t-s}2,\frac{x-y}2 \right).
\end{align}
Without loss of generality, fix $T=1$.
Since $\vf\in\cC_c^2(\Sigma_1)$, choose $\delta>0$ such that $\supp \vf\subseteq[\delta,1-\delta]^2$.
The support of $\Phi_\ve$ is then contained in
\begin{equation}
  \begin{aligned}
    t+s\in[2\delta,2-2\delta], \quad t-s\in(-2\ve,2\ve),\\
    x+y\in[2\delta,2-2\delta], \quad x-y\in(-2\ve,2\ve).
  \end{aligned}
\end{equation}
Direct computation shows that $\Phi_\ve\in\cC_c^2(\Sigma_1^2)$ for all $\ve\in(0,\delta)$.

Hereafter, we assume $\ve\in(0,\delta)$.
Fixing $(s,y)\in\Sigma_1$ and applying \eqref{eq:ei-kuz} with $k=v(s,y)$ and $\vf=\Phi_\ve(\cdot,\cdot,s,y)\in\cC_c^2(\Sigma_1)$,
\begin{equation}
  \begin{aligned}
    E_{\Phi_\ve(\cdot,\cdot,s,y)}^{(\eta,\xi)(\cdot,v(s,y))}(u) = \iint_{\Sigma_1} |u-v(s,y)|\partial_t\Phi_\ve(\cdot,\cdot,s,y)dxdt\\
    + \iint_{\Sigma_1} \xi(u,v(s,y))\partial_x\Phi_\ve(\cdot,\cdot,s,y)dxdt\\
    - \iint_{\Sigma_1} \sgn(u-v(s,y))V(x)(u-\varrho)\Phi_\ve(\cdot,\cdot,s,y)dxdt \ge 0.
  \end{aligned}
\end{equation}
Similar inequality holds for $v=v(s,y)$, $k=u(t,x)$ and $\vf=\Phi_\ve(t,x,\cdot,\cdot)$.
Denote $u_1=u(t,x)$, $v_1=v(s,y)$, then
\begin{equation}\label{eq:lem-kuz-1}
  \begin{aligned}
    \iint_{\Sigma_1} E_{\Phi_\ve(\cdot,\cdot,s,y)}^{(\eta,\xi)(\cdot,v(s,y))}(u)dsdy + \iint_{\Sigma_1} E_{\Phi_\ve(t,x,\cdot,\cdot)}^{(\eta,\xi)(\cdot,u(t,x))}(v)dxdt\\
    = \iiiint_{\Sigma_1^2} \Big\{|u_1-v_1|(\partial_t+\partial_s)\Phi_\ve + \xi(u_1,v_1)(\partial_x+\partial_y)\Phi_\ve\\
    -\sgn(u_1-v_1)\big[G(t,x,u_1)-G_*(s,y,v_1)\big]\Phi_\ve\Big\}dydsdxdt \ge 0,
  \end{aligned}
\end{equation}
where $G(t,x,u)=V(x)(u-\varrho(t,x))$ and $G_*(s,y,v)=V(y)(v-\varrho_*(s,y))$.

Introduce the coordinates $\boldsymbol\la=(\la_1,\la_2)$, $\boldsymbol\theta=(\theta_1,\theta_2)$ given by
\begin{align}
  \boldsymbol\la= \left( \frac{t+s}2,\frac{x+y}2 \right), \quad \boldsymbol\theta= \left( \frac{t-s}2,\frac{x-y}2 \right).
\end{align}
Recall that $\Phi_\ve=\vf(\boldsymbol\la)\psi_\ve(\boldsymbol\theta)$.
Direct computation shows that
\begin{equation}
  \begin{aligned}
    &(\partial_t+\partial_s)\Phi_\ve = \phi_\ve(\boldsymbol\theta)\partial_{\la_1}\vf(\boldsymbol\la),\\
    &(\partial_x+\partial_y)\Phi_\ve = \phi_\ve(\boldsymbol\theta)\partial_{\la_2}\vf(\boldsymbol\la).
  \end{aligned}
\end{equation}
Define $\Omega := \{(\boldsymbol\la,\boldsymbol\theta);\boldsymbol\la+\boldsymbol\theta\in[0,1]^2,\boldsymbol\la-\boldsymbol\theta\in[0,1]^2\}$ and
\begin{equation}
  \begin{aligned}
    &\mathcal I = |u_1-v_1|\partial_{\la_1}\vf(\boldsymbol\la) + \xi(u_1,v_1)\partial_{\la_2}\vf(\boldsymbol\la),\\
    &\mathcal G = \sgn(u_1-v_1)\big[G(\boldsymbol\la+\boldsymbol\theta,u_1)-G_*(\boldsymbol\la-\boldsymbol\theta,v_1)\big]\vf(\boldsymbol\la).
  \end{aligned}
\end{equation}
Then, \eqref{eq:lem-kuz-1} is rewritten as $\mathcal T_\ve-\mathcal R_\ve \ge 0$ for $\ve\in(0,\delta)$, where
\begin{align}
  \mathcal T_\ve = \int_\Omega \mathcal I(\boldsymbol\la,\boldsymbol\theta)\phi_\ve(\boldsymbol\theta)d(\boldsymbol\la,\boldsymbol\theta), \quad \mathcal R_\ve = \int_\Omega \mathcal G(\boldsymbol\la,\boldsymbol\theta)\phi_\ve(\boldsymbol\theta)d(\boldsymbol\la,\boldsymbol\theta).
\end{align}
Using the argument in \cite[Theorem 1]{Kruzhkov70}, one can show that
\begin{align}\label{eq:lem-kuz-2}
  \lim_{\ve\to0} \mathcal T_\ve = \iint_{\Sigma_1} \mathcal I(\boldsymbol\la,\mathbf0)d\boldsymbol\la,
\end{align}
see also \cite[Lemma 2.5.21]{MNRR96}.
Decompose $\mathcal G$ as $\mathcal G_1+\mathcal G_2+\mathcal G_3$, where
\begin{equation}
  \begin{aligned}
    \mathcal G_1 &= \sgn(u_1-v_1)\big[G(\boldsymbol\la,u_1)-G_*(\boldsymbol\la,v_1)\big]\vf(\boldsymbol\la),\\
    \mathcal G_2 &= \sgn(u_1-v_1)\big[G(\boldsymbol\la+\boldsymbol\theta,u_1)-G(\boldsymbol\la,u_1)\big]\vf(\boldsymbol\la),\\
    \mathcal G_3 &= \sgn(u_1-v_1)\big[G_*(\boldsymbol\la,v_1)-G_*(\boldsymbol\la-\boldsymbol\theta,v_1)\big]\vf(\boldsymbol\la).
  \end{aligned}
\end{equation}
Recall that $G(\boldsymbol\la,u_1)=V(\la_2)(u_1-\varrho(\boldsymbol\la))$.
Then,
\begin{equation}\label{eq:kuz-1}
  \begin{aligned}
    |\mathcal G_2| \le\,&V(\la_2)\big|\varrho(\boldsymbol\la+\boldsymbol\theta) - \varrho(\boldsymbol\la)\big|\vf(\boldsymbol\la)\\
    &+\big|V(\la_2+\theta_2)-V(\la_2)\big|\big|u_1-\varrho(\boldsymbol\la+\boldsymbol\theta)\big|\vf(\boldsymbol\la).
  \end{aligned}
\end{equation}
Since $\supp\vf \subseteq [\delta,1-\delta]^2$ and $u$, $\varrho \in L^\infty(\Sigma_1)$,
\begin{equation}\label{eq:kuz-2}
  \begin{aligned}
    \left| \int_\Omega \mathcal G_2(\boldsymbol\la,\boldsymbol\theta)\phi_\ve(\boldsymbol\theta)\,d(\boldsymbol\la,\boldsymbol\theta) \right| \le \int_{\Sigma_1} \vf(\boldsymbol\la)d\boldsymbol\la \int_{\bR^2} \phi_\ve(\boldsymbol\theta)d\boldsymbol\theta\\
    \Big\{C_\delta\big|\varrho(\boldsymbol\la+\boldsymbol\theta)-\varrho(\boldsymbol\la)\big| + C\big|V(\la_2+\theta_2)-V(\la_2)\big|\Big\},
  \end{aligned}
\end{equation}
with $C_\delta:=\sup\{V(x);\delta \le x \le 1-\delta\}$.
For sufficiently small $\ve$, both $\varrho$ and $V$ are bounded on $[\delta-\ve,1-\delta+\ve]^2$.
Then, the definition of $\phi_\ve$ and the Lebesgue differentiation theorem show that this term vanishes as $\ve\to0$.
The integral of $\mathcal G_3$ is treated similarly.
Finally,
\begin{align}
  \mathcal G_1 = |u_1-v_1|V(\la_2)\vf(\boldsymbol\la) - \sgn(u_1-v_1)V(\la_2)\big[\varrho(\boldsymbol\la)-\varrho_*(\boldsymbol\la)\big]\vf(\boldsymbol\la),
\end{align}
so that $\mathcal G_1\ge-V(\la_2)|\varrho(\boldsymbol\la)-\varrho_*(\boldsymbol\theta)|\vf(\boldsymbol\la)$.
Therefore,
\begin{align}\label{eq:lem-kuz-3}
  \liminf_{\ve\to0} \mathcal R_\ve \ge -\iint_{\Sigma_1} V(\la_2)\big|\varrho(\boldsymbol\la)-\varrho_*(\boldsymbol\la)\big|\vf(\boldsymbol\la)d\boldsymbol\la.
\end{align}
Recall that from \eqref{eq:lem-kuz-1}, we have $\mathcal T_\ve-\mathcal R_\ve\ge0$.
By \eqref{eq:lem-kuz-2} and \eqref{eq:lem-kuz-3},
\begin{align}
  \iint_{\Sigma_1} \mathcal I(\boldsymbol\la,\mathbf0)d\boldsymbol\la \ge -\iint_{\Sigma_1} V(\la_2)\big|\varrho(\boldsymbol\la)-\varrho_*(\boldsymbol\la)\big|\vf(\boldsymbol\la)d\boldsymbol\la.
\end{align}
The desired inequality follows directly.
\end{proof}

\begin{proof}[Proof of Theorem \ref{thm:uniq-non}]
First, observe that the estimate is trivial if the integral of $V|\varrho-\varrho_*|$ is infinite.
Hereafter, we assume that $\varrho-\varrho_* \in L^1(\Sigma_T;\nu)$, where $\nu(dxdt)=V(x)dxdt$.

Fix an arbitrary $\phi\in\cC_c^2((0,T))$ such that $\phi\ge0$ and recall the function $\psi_\ve$ given by \eqref{eq:aux-func}.
Using Lemma \ref{lem:kuz} with $\vf=\phi(t)\psi_\ve(x)$,
\begin{align}
  \iint_{\Sigma_T} \big[|u-v|\phi'\psi_\ve + \xi(u,v)\psi'_\ve\phi + |\varrho-\varrho_*|V\phi\psi_\ve\big]dxdt \ge 0.
\end{align}
Taking $\ve\to0$, we have
\begin{equation}
  \begin{aligned}
    &\lim_{\ve\to0} \iint_{\Sigma_T} \big(|u-v|\phi' + |\varrho-\varrho_*|V\phi\big)\psi_\ve\,dxdt\\
    =\,&\iint_{\Sigma_T} \big(|u-v|\phi' + |\varrho-\varrho_*|V\phi\big)dxdt.
  \end{aligned}
\end{equation}
Notice that the convergence of the second term follows from $\varrho-\varrho_* \in L^1(\Sigma_T;\nu)$.
We are left with the integral of $\xi(u,v)\psi'_\ve\phi$.
From the construction of $\psi'_\ve$, this term is identically $0$ for $x\in[\ve,1-\ve]$.
Using the same argument as in Lemma \ref{lem:bd}, for any $\delta>0$,
\begin{equation}
  \int_0^T \int_0^\ve \xi(u,v)\psi'_\ve\phi\,dxdt \le \frac1\delta\int_0^T \int_0^\ve \xi^2(u,v)V\,dxdt + \frac{CT\delta^2}{4\ve^2}\int_0^\ve \frac{dx}V,
\end{equation}
with a constant $C$ depending on $\phi$.
From \eqref{eq:ass-v-uniq}, the second term vanishes as $\delta\to0$, uniformly in $\ve$.
Also observe that
\begin{align}
  |\xi(u,v)| = |J(u)-J(v)| \le C|u-v|.
\end{align}
Since $u$ and $v$ satisfy \eqref{eq:bd-l2} with common boundary data $\alpha$, the first term vanishes as $\ve\to0$ for any fixed $\delta>0$.
By repeating the argument for the integral on $(1-\ve,1)$,
\begin{align}
  \lim_{\ve\to0} \iint_{\Sigma_T} \xi(u,v)\psi_\ve\phi\,dxdt = 0.
\end{align}

Putting these estimates together,
\begin{align}
  \int_0^T \left[ \phi'(t)\int_0^1 |u-v|dx + \phi(t)\int_0^1 V(x)|\varrho-\varrho_*|dx \right] dt \ge 0,
\end{align}
for all $\phi\in\cC_c^2((0,T))$ such that $\phi\ge0$.
From this,
\begin{align}
  t \mapsto \int_0^1 |u(t,\cdot)-v(t,\cdot)|dx - \iint_{\Sigma_t} V(x)|\varrho-\varrho_*|dxds
\end{align}
is an essentially decreasing function of $t$.
It suffices to apply the $L^1$-continuity of the entropy solution at $t=0$.
\end{proof}

When $V \in L^1((0,1))$, let $u$, $v$ be as in Definition \ref{defn:in} with $(\alpha,\beta,\varrho,u_0)$ and $(\alpha_*,\beta_*,\varrho_*,v_0)$, respectively.
Instead of Lemma \ref{lem:kuz}, the uniqueness follows from the next lemma.

\begin{lem}
For all $\vf\in\cC_c^2(\bR_+\times\bR)$ such that $\vf\ge0$,
\begin{equation}\label{eq:kuz-bd}
  \begin{aligned}
    &\iint_{\Sigma} \big[|u-v|\partial_t\vf + \xi(u,v)\partial_x\vf + |\varrho-\varrho_*|V\vf\big]dxdt \ge\\
    &\hspace{25mm} M\int_0^\infty \big[|\alpha-\alpha'|\vf(\cdot,0) + |\beta-\beta'|\vf(\cdot,1)\big]dt,
    \end{aligned}
\end{equation}
where the constant $M$ is the supreme of $|J'|$ between the essential infimum and supremum of $(\alpha,\alpha',\beta,\beta')$.
\end{lem}

The proof goes in the same line as that of Lemma \ref{lem:kuz}, with the boundary terms treated with the argument used in \cite[Theorem 2.7.28]{MNRR96}.
The only difference is that, when estimating $\mathcal G_2$, the support of $\vf$ contains boundary points.
Observe that $|\mathcal G_2|$ is bounded from above by
\begin{align}
  C_\vf\big[V(\la_2)|\varrho(\boldsymbol\la+\boldsymbol\theta) - \varrho(\boldsymbol\la)| + |V(\la_2+\theta_2)-V(\la_2)|\big].
\end{align}
As $V$ is integrable, almost every point in $(0,1)$ is a Lebesgue point of $V$.
This assures that the integral in \eqref{eq:kuz-2} vanishes when $\ve\to0$.

\section{Flux at boundary}
\label{sec:bd}

This section is devoted to the identification of the behavior of the entropy solution at the boundaries.
For the integrable case, \eqref{eq:bd-in} follows from \eqref{eq:ei-bd-kuz} and exactly the same argument as used in \cite[Theorem 2.7.31]{MNRR96}, so we focus on the non-integrable case and prove Theorem \ref{thm:bd}.
Hereafter, always assume \eqref{eq:ass-v-uniq} and that $\alpha$ and $\beta$ are almost everywhere constant functions.

\begin{lem}\label{lem:bd-1}
Assume \eqref{eq:ass-rho-l2}.
Let $u$ be as in Definition \ref{defn:non} and $(F,Q)$ be any boundary entropy--flux pair.
Then, for $\vf\in\cC_c^2(\bR\times(-\infty,1))$ such that $\vf\ge0$, we have $\partial_uF(u,\alpha)(u-\varrho)\vf \in L^1(\Sigma;\nu)$ and
\begin{align}
  E_\vf^{(F,Q)(\cdot,\alpha)}(u) + \int_0^1 F(u_0,\alpha)\vf(0,\cdot)dx \ge 0.
\end{align}
Similar result holds at the right boundary: for $\vf\in\cC_c^2(\bR\times(0,\infty))$ such that $\vf\ge0$, we have $\partial_uF(u,\beta)(u-\varrho)\vf \in L^1(\Sigma;\nu)$ and
\begin{align}
  E_\vf^{(F,Q)(\cdot,\beta)}(u) + \int_0^1 F(u_0,\beta)\vf(0,\cdot)dx \ge 0.
\end{align}
\end{lem}

\begin{proof}
Let $K=(-\infty,T]\times(-\infty,y]$ for some $T>0$ and $y<1$.
Denote $(f,q)=(F,Q)(\cdot,\alpha)$.
By its definition, $|f'(u)| \le C_F|u-\alpha|$.
Then,
\begin{equation}
  \begin{aligned}
    |f'(u)(u-\varrho)\mathbf1_K| &\le C_F|(u-\alpha)(u-\varrho)\mathbf1_K|\\
    &\le C_F\big(|(u-\varrho)^2\mathbf1_K| + |(\varrho-\alpha)(u-\varrho)\mathbf1_K|\big).
  \end{aligned}
\end{equation}
By \eqref{eq:eb} and \eqref{eq:ass-rho-l2}, both terms belong to $L^1(\Sigma;\nu)$.
Hence, for all $\vf\in\cC_c^2(\bR\times(-\infty,1))$, $f'(u)(u-\varrho)\vf \in L^1(\Sigma;\nu)$.

Fix $\vf\in\cC_c^2(\bR\times(-\infty,1))$ such that $\vf\ge0$.
Let $\vf_\ve=\vf\psi_\ve$, where $\psi_\ve=\psi_\ve(x)$ is given by \eqref{eq:aux-func}.
Since $\alpha$ is constant, from \eqref{eq:ei},
\begin{align}
  E_{\vf_\ve}^{(f,q)}(u)  + \int_0^1 f(u_0)\vf_\ve(0,\cdot)dx \ge 0.
\end{align}
Taking $\ve\to0$, it is straightforward to see that
\begin{equation}
  \begin{aligned}
    \lim_{\ve\to0} \int_0^1 f(u_0)\vf_\ve(0,\cdot)dx + \iint_\Sigma f(u)\partial_t\vf_\ve\,dxdt\\
    = \int_0^1 f(u_0)\vf(0,\cdot)dx + \iint_\Sigma f(u)\partial_t\vf\,dxdt.
  \end{aligned}
\end{equation}
Using Lemma \ref{lem:bd}, since $q(\alpha)=Q(\alpha,\alpha)=0$ and $\vf(t,1)=0$,
\begin{align}
  \lim_{\ve\to0} \iint_\Sigma q(u)\partial_x\vf_\ve\,dxdt = \iint_\Sigma q(u)\partial_x\vf\,dxdt.
\end{align}
Recall that $G(\cdot,\cdot,u)=V(x)(u-\varrho)$ and $f'(u)(u-\varrho)\vf \in L^1(\Sigma;\nu)$, the dominated convergence theorem yields that
\begin{align}
  \lim_{\ve\to0} \iint_\Sigma f'(u)G(\cdot,\cdot,u)\vf_\ve\,dxdt = \iint_\Sigma f'(u)G(\cdot,\cdot,u)\vf\,dxdt.
\end{align}
Putting them together, we obtain the first assertion in the lemma.
The second one follows similarly.
\end{proof}

To continue, we make use of the condition \eqref{eq:ass-rho-l1} to refine the energy bound \eqref{eq:eb} to the following $L^1$-integrability.

\begin{prop}\label{prop:l1}
Assume \eqref{eq:ass-rho-l1}, then $u-\varrho \in L^1(\Sigma_T;\nu)$, i.e.,
\begin{align}
  \iint_{\Sigma_T} V(x)|u(t,x)-\varrho(t,x)|\,dxdt < \infty, \quad \forall\,T>0.
\end{align}
\end{prop}

\begin{proof}
Thanks to \eqref{eq:ass-rho-l1}, it suffices to prove for $y\in(0,1)$ that
\begin{align}\label{eq:l1-1}
  \int_0^T \int_0^y V(x)|u(t,x)-\alpha|\,dxdt < \infty,
\end{align}
and the similar bound for $\beta$.
Recall the functions $(F_\ve,Q_\ve)(\cdot,k)$ defined in \eqref{eq:kuz-appro} and observe that $(F_\ve,Q_\ve)$ forms a boundary entropy--flux pair for fixed $\ve$.
Pick some $T_*>T$ and $y_*\in(y,1)$, the previous lemma yields that
\begin{align}
  E_\vf^{(F_\ve,Q_\ve)(\cdot,\alpha)}(u) + \int_0^1 F_\ve(u_0,\alpha)\vf(0,\cdot)dx \ge 0,
\end{align}
for all $\vf\in\cC_c^2((-\infty,T_*)\times(-\infty,y_*))$ such that $\vf\ge0$.
Hence,
\begin{align}
  \sup_{\ve>0} \iint_\Sigma G(\cdot,\cdot,u)\partial_uF_\ve(u,\alpha)\vf\,dxdt < \infty.
\end{align}
Fix such a $\vf$ and decompose $G(\cdot,\cdot,u)\partial_uF_\ve(u,\alpha)\vf$ to
\begin{align}
  V\partial_uF_\ve(u,\alpha)(u-\alpha)\vf - V\partial_uF_\ve(u,\alpha)(\varrho-\alpha)\vf.
\end{align}
Since $|\partial_uF_\ve|$ is bounded by $1$ uniformly in $\ve$, by \eqref{eq:ass-rho-l1},
\begin{equation}
  \begin{aligned}
    &\left| \iint_\Sigma V\partial_uF_\ve(u,\alpha)(\varrho-\alpha)\vf\,dxdt \right|\\
    \le\,&\|\vf\|_{L^\infty(\Sigma)} \int_0^{T_*} \int_0^{y_*} V(x)|\varrho(t,x)-\alpha|\,dxdt
  \end{aligned}
\end{equation}
is bounded from above uniformly in $\ve$.
Therefore,
\begin{align}
  \sup_{\ve>0} \iint_\Sigma V\partial_uF_\ve(u,\alpha)(u-\alpha)\vf\,dxdt < \infty.
\end{align}
From the construction of $F_\ve$, $\partial_uF(u,\alpha)(u-\alpha)\ge0$ for $u\in\bR$ and $\partial_uF_\ve(u,\alpha)=\sgn(u-\alpha)$ if $|u-\alpha|>2\ve$.
Then, for each fixed $\ve$,
\begin{align}
  \mathbf1_{\{|u-\alpha|>2\ve\}} V|u-\alpha|\vf \le V\partial_uF_\ve(u,\alpha)(u-\alpha)\vf,
\end{align}
and in consequence,
\begin{align}
  \sup_{\ve>0} \iint_\Sigma \mathbf1_{\{|u-\alpha|>2\ve\}} V|u-\alpha|\vf\,dxdt < \infty.
\end{align}
Monotonic convergence theorem then yields that
\begin{align}
  \iint_\Sigma V(x)|u(t,x)-\alpha|\vf(t,x)dxdt < \infty.
\end{align}
The proof is concluded by choosing $\vf$ such that $\vf|_{(0,T)\times(0,y)}\equiv1$.
\end{proof}

\begin{rem}
Assume the conditions in Theorem \ref{thm:bd}.
Due to Proposition \ref{prop:l1}, \eqref{eq:ei} in Definition \ref{defn:in} can be generalized to
\begin{equation}
  \begin{aligned}
    &E_\vf^{(f,q)}(u) + \int_0^1 f(u_0)\vf(0,\cdot)dx\\
    \ge\;&q(\beta)\int_0^\infty \vf(t,1)dt - q(\alpha)\int_0^\infty \vf(t,0)dt,
  \end{aligned}
\end{equation}
for all Lax entropy--flux pairs $(f,q)$ and $\vf\in\cC_c^2(\bR^2)$ such that $\vf\ge0$.
The same generalization works for \eqref{eq:ei-kuz}.
\end{rem}


Now we can state the proof of Theorem \ref{thm:bd}.

\begin{proof}
Pick nonnegative functions $\phi\in\cC_c^2((0,T))$, $\psi\in\cC_c^2(\bR)$ and define $\vf=\phi(t)\psi(x)$.
From the previous remark,
\begin{align}
  E_\vf^{(f,q)}(u) + \big[\psi(0)q(\alpha)-\psi(1)q(\beta)\big]\int_0^T \phi(t)dt \ge 0.
\end{align}
There is a constant $C=C(f,\phi)$, such that
\begin{align}
  E_\vf^{(f,q)}(u) \le \iint_{\Sigma_T} \big[q(u)\psi'\phi + C(V|u-\varrho|+1)\psi\big]dxdt.
\end{align}
Due to the integrability proved in Proposition \ref{prop:l1},
\begin{equation}
  \begin{aligned}
    F_{q,\phi}(x) := &\int_0^\infty q(u(t,x))\phi(t)dt\\
    &- C\int_0^T \int_0^x \big(V(y)|u(t,y)-\varrho(t,y)|+1\big)dy,
  \end{aligned}
\end{equation}
is well-defined as a measurable function on $(0,1)$.
If $\psi(1)=0$,
\begin{align}
  \psi(0)q(\alpha)\int_0^T \phi(t)dt + \int_0^1 F_{q,\phi}(x)\psi'(x)dx \ge 0.
\end{align}
This holds for all nonnegative $\psi\in\cC_c^2((-\infty,1))$, so $F_{q,\phi}$ is non-increasing after possible modification on a Lebesgue null subset of $(0,1)$.
Hence,
\begin{align}
  \esslim_{x\to0+} F_{q,\phi}(x) = \esslim_{x\to0+} \int_0^\infty q(u(t,x))\phi(t)dt
\end{align}
exists for all $\phi\in\cC_c^\infty((0,T))$ such that $\phi\ge0$.
For all $(s,t) \subseteq (0,T)$, the result extends to $\phi=\mathbf1_{(s,t)}$ with standard argument.
The first equation in \eqref{eq:bd} then follows from \eqref{eq:l1-1} and the fact that $V$ is not integrable on any neighbor of $0$.
The second one is proved similarly.
\end{proof}

When $J$ is convex or concave, more information can be extracted from \eqref{eq:bd} by exploiting the idea in \cite{MOX21,DMOX22}.

\begin{proof}[Proof of Corollary \ref{cor:bd}]
Let $\mathcal Q$ be a countable set of functions such that \eqref{eq:bd} holds.
The choice of $\mathcal Q$ will be specified later.
Fix an interval $(s,t)$, there exists a subset $\mathcal E \subseteq (0,1)$ with Lebesgue measure $0$, such that
\begin{enumerate}
\item[(\romannum1)] $\|u(\cdot,x)\|_{L^\infty((s,t))} \le \|u\|_{L^\infty((s,t)\times(0,1))}$ for all $x\in(0,1)\backslash\mathcal E$;
\item[(\romannum2)]
$(t-s)q(\alpha) = \lim_{x\in(0,1)\backslash\mathcal E,x\to0+} \int_{(s,t)} q(u(r,x))dr$ for all $q\in\mathcal Q$.
\end{enumerate}
Denote $m=\|u\|_{L^\infty((s,t)\times(0,1))}$.
For any sequence $x_n\in(0,1)\backslash\mathcal E$ such that $x_n\to0$, we can find a subsequence $x'_n$ and a family $\{\mu_r\}_{r\in(s,t)}$ of probability measures, such that $\mu_r([-m,m])=1$ and for each $q\in\mathcal Q$,
\begin{align}\label{eq:bd-1-1}
  (t-s)q(\alpha) = \int_s^t \int_\bR q(z)\mu_t(dz)dr.
\end{align}
In other words, $u(\cdot,x'_n)$ converges to $\{\mu_r\}$ as $n\to\infty$ in the weak-$\star$ topology of $L^\infty((s,t))$.
To show \eqref{eq:bd-1}, we need to show that
\begin{align}\label{eq:bd-1-2}
  \mu_r(\{\alpha\})=1 \quad \text{for almost all}\ r\in(s,t).
\end{align}
For each rational number $\delta$, define
\begin{equation}
  \begin{aligned}
    &f_{-,\delta}(u) := \mathbf1_{u\le\delta}|u-\delta|, \quad q_{-,\delta}(u) := \mathbf1_{u\le\delta}(J(\delta)-J(u)),\\
    &f_{+,\delta}(u) := \mathbf1_{u\ge\delta}|u-\delta|, \quad q_{+,\delta}(u) := \mathbf1_{u\ge\delta}(J(u)-J(\delta)).
  \end{aligned}
\end{equation}
It is easy to show that we can choose $\mathcal Q$ to contain all $q_{\pm,\delta}$, so \eqref{eq:bd-1-1} holds for them.
Observe that \eqref{eq:bd-1-2} is straightforward if $J$ is monotonically increasing (or decreasing) on $[-m,m]$.
Indeed, suppose that $J'\ge0$ on $[-m,m]$.
For $\delta<\alpha$, $q_{-,\delta}(u)>q_{-,\delta}(\alpha)$ for $u<\delta$ and $q_{-,\delta}(u)=q_{-,\delta}(\alpha)$ for $u\ge\delta$.
Therefore, $\mu_r([-m,\delta))=0$.
Similarly, $\mu_r((\delta,m])=0$ for $\delta>\alpha$.
As $\delta$ can be any rational number, \eqref{eq:bd-1-2} holds.
The case $J$ is decreasing is similar.

Hereafter, we assume that $J$ is concave and attaches its maximum at $m_*\in[-m,m]$.
Suppose that $\alpha \le m_*$, by the argument above
\begin{align}\label{eq:bd-1-3}
  \mu_r\big([\alpha,\alpha_*]\big)=1 \quad \text{for almost all}\ r\in(s,t),
\end{align}
where $\alpha_*>\alpha$ is the only point that $J(\alpha)=J(\alpha')$.
For $\delta>\alpha_*$, $q_{-,\delta}(u) \le q_{-\delta}(\alpha)$ on $[\alpha,\alpha_*]$ with equality holds only for $u=\alpha$, $\alpha_*$.
Therefore, \eqref{eq:bd-1-3} holds with $[\alpha,\alpha_*]$ is replaced by $\{\alpha,\alpha_*\}$.
Finally, let $\mathcal Q$ also contain some Lax flux $q$ such that $q(\alpha_*)>q(\alpha)$ strictly, so \eqref{eq:bd-1-2} holds.
\end{proof}

\section{Existence of the entropy solution}
\label{sec:exis}

In this section, we fix some $T>0$ and construct an entropy solution on $\Sigma_T$ via the vanishing viscosity limit.
With the uniqueness proved in Theorem \ref{thm:uniq}, we obtain an entropy solution on $\Sigma$.
As before, we focus on the non-integrable case and then summarize the argument for the integrable case.

For the non-integrable case, assume that the conditions in Theorem \ref{thm:exis} hold.
For each $\ve>0$, the viscosity problem is constructed as
\begin{equation}\label{eq:bl-vis}
  \left\{
  \begin{aligned}
    &\partial_tu^\ve + \partial_x[J(u^\ve)] + G^\ve(t,x,u^\ve) = \ve\partial_x^2u^\ve, \quad (t,x)\in\Sigma_T,\\
    &u^\ve(0,x)=u_0^\ve(x), \quad u^\ve(t,0)=\alpha(t), \quad u^\ve(t,1)=\beta(t),
  \end{aligned}
  \right.
\end{equation}
where $G^\ve(t,x,u):=V(x)(u-\varrho^\ve(t,x))$ with $\varrho^\ve$ in Theorem \ref{thm:exis}, $u_0^\ve\in\cC^2([0,1])$ approximates $u_0$ in $L^2((0,1))$ and
\begin{align}
  u_0^\ve(0)=\alpha(0), \quad u_0^\ve(1)=\beta(0).
\end{align}
It admits a classical solution $u^\ve=u^\ve(t,x)$ that satisfies
\begin{itemize}
\item[\textnormal{(v1)}] $u^\ve-\varrho^\ve \in L^2(\Sigma_T;\nu)$, and
\item[\textnormal{(v2)}] for all $\vf\in\cC_c^2((-\infty,T)\times(0,1))$,
\begin{equation}
  \begin{aligned}
    \int_0^1 u_0^\ve\vf(0,\cdot)dx + \iint_{\Sigma_T} \big[u^\ve\partial_t\vf + \ve u^\ve\partial_x^2\vf + J(u^\ve)\partial_x\vf\big]dxdt\\
    = \iint_{\Sigma_T} G(\cdot,\cdot,u^\ve)\vf\,dxdt.
  \end{aligned}
\end{equation}
\end{itemize}
Some useful properties of $u^\ve$ are collected in Appendix \ref{sec:vis}.

\begin{thm}\label{thm:exis-non}
Along proper subsequence of $\ve\to0$, $u^\ve$ converges to some $u \in L^\infty(\Sigma_T)$ with respect to the weak-$\star$ topology of $L^\infty(\Sigma_T)$.
Furthermore, the limit point satisfies \textnormal{(EB)} for the given $T$ and \textnormal{(EI)} for all Lax entropy--flux pairs $(f,q)$ and $\vf\in\cC_c^2((-\infty,T)\times\bR)$ such that $\vf\ge0$.
\end{thm}

Recall that a Young measure $\mu=\{\mu_{t,x};(t,x)\in\Sigma_T\}$ is a family of probability measures on $\bR$ such that $(t,x) \mapsto \mu_{t,x}(A)$ is a measurable map from $\Sigma_T$ to $[0,1]$ for any Borel subset $A$ of $\bR$.
For continuous function $h$, define
\begin{align}
  \bar h: \Sigma_T \ni (t,x) \mapsto \int_0^1 h(z)\mu_{t,x}(dz).
\end{align}
In view of Lemma \ref{lem:vis-uniform-bd}, $\|u^\ve\|_{L^\infty(\Sigma_T)}$ is uniformly bounded.
According to the fundamental theorem of Young measure, we obtain a $\mu=\{\mu_{t,x};(t,x)\in\Sigma_T\}$ as a subsequential limit point of $u^\ve$ in the following sense: for all $h\in\cC(\bR)$ and $\vf \in L^1(\Sigma_T)$,
\begin{align}\label{eq:young}
  \lim_{\ve\to0} \iint_{\Sigma_T} h(u^\ve)\vf(t,x)dxdt = \iint_{\Sigma_T} \bar h(t,x)\vf(t,x)dxdt.
\end{align}
We also have $\mu_{t,x}([-m,m])=1$, where $m=\sup_{\ve>0} \|u^\ve\|_{L^\infty(\Sigma_T)}$.

\begin{proof}[Proof of Theorem \ref{thm:exis-non}]
First, from Lemma \ref{lem:vis-eb} and \cite[Proposition 4.1]{BertheV19},
\begin{align}\label{eq:eb-young}
  \iint_{\Sigma_T} V(x) \left[ \int_0^1 \big[z-\varrho(t,x)\big]^2\mu_{t,x}(dz) \right] dxdt< \infty.
\end{align}
For all Lax entropy--flux pairs $(f,q)$, from \eqref{eq:bl-vis} we have
\begin{equation}
  \begin{aligned}
    \partial_t[f(u^\ve)] + \partial_x[q(u^\ve)] &= f'(u^\ve)\big\{\partial_tu^\ve+\partial_x[J(u^\ve)]\big\}\\
    &= \ve f'(u^\ve)\partial_x^2u^\ve - f'(u^\ve)G^\ve(\cdot,\cdot,u^\ve).
  \end{aligned}
\end{equation}
Since $f''\ge0$, $f'(u^\ve)\partial_x^2u^\ve \le \ve\partial_x^2[f(u^\ve)]$.
Therefore,
\begin{align}
  \partial_t[f(u^\ve)] + \partial_x[q(u^\ve)] + f'(u^\ve)G^\ve(\cdot,\cdot,u^\ve) \le \ve\partial_x^2[f(u^\ve)].
\end{align}
Recall the entropy product defined in \eqref{eq:ent-prod}.
For $\vf\in\cC_c^2((-\infty,T)\times(0,1))$ such that $\vf\ge0$, we have
\begin{equation}
  \begin{aligned}
    E_\vf^{(f,q)}(u^\ve) + \int_0^1 f(u_0^\ve)\vf(0,\cdot)dx \ge\\
    \ve\iint_{\Sigma_T} \partial_x[f(u^\ve)]\partial_x\vf\,dxdt - \iint_{\Sigma_T} f'(u^\ve)V(\varrho^\ve-\varrho)\vf\,dxdt.
  \end{aligned}
\end{equation}
In view of the condition (\romannum1) in Theorem \ref{thm:exis} and Lemma \ref{lem:vis-eb}, the two terms in the right-hand side vanish as $\ve\to0$.
We then obtain from \eqref{eq:young} that
\begin{equation}\label{eq:ei-young}
  \begin{aligned}
  &\iint_{\Sigma_T} \left[ \bar f\partial_t\vf + \bar q\partial_x\vf - \big(\bar g-\overline{f'}\varrho\big)V\vf \right] dxdt\\
  \ge\,&- \int_0^1 f(u_0)\vf(0,\cdot)dx, \quad \text{where}\ g(u):=uf'(u).
  \end{aligned}
\end{equation}

Observe that \eqref{eq:eb-young} and \eqref{eq:ei-young} can be viewed as the measure-valued version of \eqref{eq:eb} and \eqref{eq:ei}, respectively.
Hence, the main task is to show that, the Young measure $\mu$ is concentrated on some $u \in L^\infty(\Sigma_T)$:
\begin{align}\label{eq:delta-young}
  \mu_{t,x}(dz) = \delta_{u(t,x)}(dz) \quad \text{for almost all}\ (t,x)\in\Sigma_T.
\end{align}
To do this, we exploit the \emph{compensated compactness} argument, see, e.g., \cite[Section 5.D]{Evans90}.
Define two sequences $\Phi_\ve$, $\Psi_\ve: \Sigma_T\to\bR^2$ by
\begin{align}
  \Phi_\ve:=(f(u^\ve),q(u^\ve)), \quad \Psi_\ve:=(-J(u^\ve),u^\ve).
\end{align}
Since $\{\Phi_\ve; \ve>0\}$ and $\{\Psi_\ve; \ve>0\}$ are bounded, $\{\mathrm{div}\Phi_\ve; \ve>0\}$ and $\{\mathrm{curl}\Psi_\ve; \ve>0\}$ are bounded in $W^{-1,p}(\Sigma_T)$ for any $p>2$.
Notice that
\begin{equation}
  \begin{aligned}
    \mathrm{div}\Phi_\ve &= \partial_t[f(u^\ve)]+\partial_x[q(u^\ve)]\\
    &= \ve\partial_x^2[f(u^\ve)]-\ve f''(u^\ve)(\partial_xu^\ve)^2-Vf'(u^\ve)(u^\ve-\varrho^\ve);\\
    \mathrm{curl}\Psi_\ve &= \partial_tu^\ve+\partial_x[J(u^\ve)] = \ve\partial_x^2u^\ve-V(u^\ve-\varrho^\ve).
  \end{aligned}
\end{equation}

Fix any $\delta>0$ and define $\Sigma_T^\delta = (\delta,T-\delta)\times(\delta,1-\delta)$.
We claim that both $\{\mathrm{div}\Phi_\ve;\ve>0\}$ and $\{\mathrm{curl}\Psi_\ve;\ve>0\}$ are precompact in $H^{-1}(\Sigma_T^\delta)$.
Indeed, we have seen from Lemma \ref{lem:vis-eb} that $\ve\partial_x^2[f(u^\ve)]$ vanishes as $\ve\to0$ in $H^{-1}(\Sigma_T^\delta)$ and $\{\ve f''(u^\ve)(\partial_xu^\ve)^2;\ve>0\}$ is a bounded sequence in $L^1(\Sigma_T^\delta)$.
On the other hand, as $V \le C_\delta$ on $[\delta,1-\delta]$, $\{Vf'(u^\ve)(u^\ve-\varrho^\ve);\ve>0\}$ is also a bounded sequence in $L^1(\Sigma_T^\delta)$.
Thanks to \cite[Corollary 1.C.1]{Evans90}, the claim holds for $\{\mathrm{div}\Phi_\ve;\ve>0\}$.
For $\{\mathrm{curl}\Psi_\ve;\ve>0\}$, the argument is similar.

Now, the Div-Curl lemma \cite[Theorem 5.B.4]{Evans90} yields that
\begin{align}
  \lim_{\ve\to0} \big(\Phi_\ve \cdot \Psi_\ve\big) = (\bar f,\bar q) \cdot (-\bar J,\bar u) = \bar u\bar q-\bar J\bar f,
\end{align}
weakly as distributions on $\Sigma_T^\delta$.
Meanwhile, \eqref{eq:young} with $h=zq(z)-J(z)f(z)$ gives us that for all $\vf \in L^1(\Sigma_T^\delta)$,
\begin{align}
  \lim_{\ve\to0} \iint_{\Sigma_T^\delta} \big(\Phi_\ve \cdot \Psi_\ve\big)\vf\,dxdt = \iint_{\Sigma_T^\delta} \bar h\vf\,dxdt.
\end{align}
Hence, the \emph{Tartar's factorization} holds almost everywhere in $\Sigma_T^\delta$:
\begin{align}\label{eq:tartar}
  \int_0^1 (J-\bar J)(f-\bar f)d\mu_{t,x}=\int_0^1 (z-\bar u)(q-\bar q)d\mu_{t,x}.
\end{align}
As $\delta>0$ is arbitrary, we obtain \eqref{eq:tartar} for all Lax entropy--flux pairs $(f,q)$ and almost all $(t,x)\in\Sigma_T$.
Standard argument then proves \eqref{eq:delta-young}.
\end{proof}

For the integrable case, the approach is slightly different.
Assume that $V \in L^1((0,1))$ and $(\varrho,\alpha,\beta,u_0)$ are essentially bounded functions.
For $\ve>0$, pick $\varrho^\ve\in\cC^2(\Sigma_T)$, $\alpha^\ve$, $\beta^\ve\in\cC^2([0,T])$ and $u_0^\ve\in\cC^2([0,1])$ as a mollification of $\varrho$, $\alpha$, $\beta$ and $u_0$:
\begin{equation}
  \begin{aligned}
    \lim_{\ve\to0} \bigg\{ \int_0^1 (u_0^\ve-u_0)^2dx + \iint_{\Sigma_T} V(\varrho^\ve-\varrho)^2dxdt \hspace{10mm}\\
    +\int_0^T \big[(\alpha^\ve-\alpha)^2 + (\beta^\ve-\beta)^2\big]dt \bigg\} = 0,
  \end{aligned}
\end{equation}
and for each $\ve>0$,
\begin{equation}
  \begin{aligned}
    &\varrho^\ve(t,0)=\alpha^\ve(t), \quad \varrho^\ve(t,1)=\beta^\ve(t), \quad \forall\,t\in[0,T],\\
    &u_0^\ve(0)=\alpha^\ve(0), \quad u_0^\ve(1)=\beta^\ve(0).
  \end{aligned}
\end{equation}
The viscosity problem for integrable case reads
\begin{equation}
  \left\{
  \begin{aligned}
    &\partial_tu^\ve + \partial_x[J(u^\ve)] + G^\ve(t,x,u^\ve) = \ve\partial_x^2u^\ve,\\
    &u^\ve(0,x)=u_0^\ve(x), \quad u^\ve(t,0)=\alpha^\ve(t), \quad u^\ve(t,1)=\beta^\ve(t).
  \end{aligned}
  \right.
\end{equation}
where $G^\ve(t,x,u):=V(x)(u-\varrho^\ve(t,x))$.

Let $u^\ve$ be the classical solution and consider the limit $\ve\to0$ as in the non-integrable case.
To deal with the discontinuities formulated at the boundaries in this limit procedure, define for each $\ve>0$ that
\begin{align}
  g_\ve(x):=
  \begin{cases}
    1-e^{-\frac x\ve}, &x\in[0,\frac12),\\
    1-e^{-\frac{1-x}\ve}, &x\in(\frac12,1].
  \end{cases}
\end{align}
For boundary entropy--flux $(F,Q)$ and $k\in\bR$, denote $(f,q)=(F,Q)(\cdot,k)$.
For $\vf\in\cC_c^2((-\infty,T)\times\bR)$, let $\vf_\ve=\vf g_\ve$ and observe that
\begin{equation}
  \begin{aligned}
    \mathcal E(\ve) := E_{\vf_\ve}^{(f,q)}(u^\ve) - \iint_{\Sigma_T} q(u^\ve)\vf g'_\ve\,dxdt =\\
    \iint_{\Sigma_T} \big[f(u^\ve)\partial_t\vf+q(u^\ve)\partial_x\vf-f'(u^\ve)G(\cdot,\cdot,u^\ve)\big]g_\ve\,dxdt.
  \end{aligned}
\end{equation}
Following the manipulation in \cite[Theorem 2.8.4]{MNRR96}, we show that
\begin{equation}
  \begin{aligned}
    \liminf_{\ve\to0} \left\{ \int_0^1 f(u_0)\vf(0,\cdot)dx + \mathcal E(\ve) \right\} \hspace{25mm}\\
    \ge -\,M\int_0^T \big[f(\alpha)\vf(\cdot,0) + f(\beta)\vf(\cdot,1)\big]dt,
  \end{aligned}
\end{equation}
when $\vf\ge0$.
From this, we obtain the measure-valued version of \eqref{eq:ei-bd} for the subsequential weak-$\star$ limit of $u^\ve$.
The application of compensated compactness argument is exactly the same as in the non-integrable case.

\appendix

\section{The parabolic problem with non-integrable $V$}
\label{sec:vis}

Assume the conditions in Theorem \ref{thm:exis}.
Without loss of generality, assume that $0 \le \varrho^\ve$, $u_0^\ve \le 1$ on $[0,T]\times[0,1]$.
Let $u^\ve$ be the solution to the parabolic equation \eqref{eq:bl-vis}.
We collect and prove some useful estimates for $u^\ve$.

\begin{lem}\label{lem:vis-uniform-bd}
For all $\ve>0$ and $(t,x)\in\Sigma_T$, $0 \le u^\ve \le 1$.
\end{lem}

\begin{proof}
Use the short notation $\langle \cdot,\cdot \rangle$ to denote the inner product in $L^2(\Sigma_T)$.
First assume that $J$ is globally Lipschitz continuous:
\begin{equation*}
  |J(u)-J(u')| \le M|u-u'|.
\end{equation*}
Let $v=(u^\ve-1)^+$ and note that $v|_{t=0}=v|_{x=0,1}=0$.
Hence,
\begin{equation*}
  \begin{aligned}
    \langle \partial_tu^\ve, v \rangle &= \langle \partial_tv, v \rangle = \frac12\int_0^1 v^2(T,x)dx,\\
    \langle \partial_x[J(u^\ve)], v \rangle &= -\langle J(u)-J(1), \partial_xv \rangle \ge -M\langle |u-1|, |\partial_xv| \rangle\\
    &= -M\langle v, |\partial_xv| \rangle \ge -\ve\|\partial_xv\|_{L^2(\Sigma_T)}^2 - \frac M{4\ve} \|v\|_{L^2(\Sigma_T)}^2,\\
    \ve\langle \partial_x^2u^\ve, v \rangle &= \ve\langle \partial_x^2v, v \rangle = -\ve\|\partial_xv\|_{L^2}^2.
  \end{aligned}
\end{equation*}
Testing \eqref{eq:bl-vis} with $v$, as $\langle G(\cdot,\cdot,u),v \rangle \ge 0$,
\begin{equation*}
  \int_0^1 v^2(T,x)dx \le \frac M{2\ve} \|v\|_{L^2(\Sigma_T)}^2 = \frac M{2\ve}\iint_{\Sigma_T} v^2(t,x)dxdt.
\end{equation*}
\Gro inequality shows that $v(T,\cdot)\equiv0$, i.e., $u^\ve(T,\cdot)\le1$ almost everywhere.
Similar argument with $v=(u^\ve)^-$ shows that $u^\ve(T,\cdot)\ge0$.
Since $T$ appeared above can be replaced with any $t\in(0,T)$, $0 \le u^\ve(t,x) \le 1$ for all $(t,x)\in\Sigma_T$.

If $J$ is not globally Lipschitz continuous, construct $J_*$, such that $J_*=J$ on $[0,1]$ and $J_*$ is globally Lipschitz continuous.
The above proof shows that we can replace $J$ with $J_*$ and the solution would not be affected.
\end{proof}

\begin{lem}\label{lem:vis-eb}
Recall that $\|\cdot\|_{L^2(\Sigma_T)}$ is the $L^2$ norm with respect to the Lebesgue measure, while $\|\cdot\|_{L^2(\Sigma_T;\nu)}$ is the $L^2$ norm with respect to $d\nu=V(x)dxdt$.
\begin{align}\label{eq:vis-eb}
  \sup_{\ve>0} \left\{ \ve\|\partial_xu^\ve\|_{L^2(\Sigma_T)}^2 + \|u^\ve-\varrho\|_{L^2(\Sigma_T;\nu)}^2 \right\} < \infty.
\end{align}
\end{lem}

\begin{proof}
First, from condition (\romannum2) in Theorem \ref{thm:exis}, the continuous function $\varrho^\ve$ satisfies \eqref{eq:ass-rho-l2}, so that $\varrho^\ve(t,0)=\alpha(t)$, $\varrho^\ve(t,1)=\beta(t)$ for all $t\in[0,T]$.
Also, thanks to (\romannum2), it suffices to verify \eqref{eq:vis-eb} with $\varrho$ replaced by $\varrho^\ve$.

Let $w^\ve=u^\ve-\varrho^\ve$.
As both $u^\ve$ and $\varrho^\ve$ are uniformly bounded,
\begin{equation*}
  \begin{aligned}
    \langle \partial_tu^\ve, w^\ve \rangle &= \frac12\int_0^1 (u^\ve-\varrho^\ve)^2\big|_{t=0}^{t=T}\,dx + \langle \partial_t\varrho^\ve, u^\ve-\varrho^\ve \rangle\\
    &\ge -\,C(1 + \|\partial_t\varrho^\ve\|_{L^2(\Sigma_T)}),\\
    \langle \partial_x[J(u^\ve)], w^\ve \rangle &= \int_0^T [\mathcal J(\cdot,\cdot,u^\ve)]\big|_{x=0}^{x=1}\,dt + \langle \partial_x\varrho^\ve, J(u^\ve) \rangle\\
    &\ge -\,C(1+\|\partial_x\varrho^\ve\|_{L^2(\Sigma_T)}),
  \end{aligned}
\end{equation*}
where $\mathcal J=\mathcal J(t,x,u)$ is given by
\begin{align*}
  \mathcal J(t,x,u) := \int_0^u wJ'(w)dw-\varrho^\ve(t,x)J(u).
\end{align*}
Noting that $w^\ve|_{x=0,1}=0$,
\begin{equation*}
  \begin{aligned}
    \langle \partial_x^2u^\ve, w^\ve \rangle = \langle \partial_xu^\ve, - \partial_xw^\ve \rangle &= \langle \partial_xu^\ve,\partial_x\varrho^\ve \rangle - \|\partial_xu^\ve\|_{L^2(\Sigma_T)}^2\\
    &\le \frac12\|\partial_x\varrho^\ve\|_{L^2(\Sigma_T)}^2 - \frac12\|\partial_xu^\ve\|_{L^2(\Sigma_T)}^2.
  \end{aligned}
\end{equation*}
Testing the equation with $w^\ve$, we get
\begin{equation*}
  \frac\ve2\|\partial_xu^\ve\|_{L^2(\Sigma_T)}^2 + \iint_{\Sigma_T} V(x)(w^\ve)^2dxdt \le C(1+\|\varrho^\ve\|_{H^1(\Sigma_T)}).
\end{equation*}
The estimate holds since $\varrho^\ve$ is uniformly bounded in $H^1(\Sigma_T)$.
\end{proof}

\titleformat{\section}[hang]	
{\bfseries\large}{}{0em}{}[]
\titlespacing*{\section}{0em}{2em}{1.5em}

\titleformat{\subsection}[runin]
{\bfseries\normalsize}{}{0em}{}[.]

\section{Declarations}

\subsection{Acknowledgements}
The author greatly thanks Stefano Modena for the inspiring discussion and advices.

%
%
%

\section{References}	
\renewcommand{\section}[2]{}
\bibliography{bibliography.bib}



\vspace{2em}
\noindent{\large Lu \textsc{Xu}}

\vspace{0.5em}
\noindent Gran Sasso Science Institute\\
Viale Francesco Crispi 7, L'Aquila, Italy\\
{\tt lu.xu@gssi.it}

\end{document}